\documentclass[11pt]{amsart}

\usepackage{a4wide}
\usepackage[usenames,dvipsnames,svgnames,table]{xcolor}
\usepackage[colorlinks=true, pdfstartview=FitV, linkcolor=ForestGreen,citecolor=ForestGreen, urlcolor=blue]{hyperref}
\usepackage{amsfonts}
\usepackage{amssymb}
\usepackage{amsmath}
\usepackage{amsthm}
\usepackage{latexsym}
\usepackage{color}
\usepackage{amscd}
\usepackage[all,cmtip]{xy}
\usepackage{graphicx}
\usepackage{enumitem} 
\usepackage[margin=1.6in]{geometry}
\usepackage{esint}

\definecolor{labelkey}{rgb}{0,0,1}

\usepackage{epsfig}
\allowdisplaybreaks

\numberwithin{equation}{section}
\newcounter{exercise}

\newtheorem{theorem}{Theorem}

\newtheorem{lemma}[theorem]{Lemma}

\theoremstyle{definition}

\theoremstyle{remark}
\newtheorem{remark}[theorem]{Remark}

\newcommand\N{{\mathbb N}}
\newcommand\R{{\mathbb R}}

\newcommand{\cQ}{\mathcal Q}
\newcommand{\cR}{\mathcal R}

\def\eps{{\varepsilon}}

\DeclareMathOperator{\supp}{supp}

\DeclareMathOperator{\osc}{osc}

\newcommand{\dd}{{\, \mathrm d}}

\newcommand{\n}[1]{{\left\| #1 \right\|}}

\newcommand{\mk}{\medskip}

\let\oldmarginpar\marginpar
\renewcommand\marginpar[1]{\-\oldmarginpar[\raggedleft\footnotesize #1]%
{\raggedright\footnotesize #1}}

\def\signci{\bigskip \begin{center} {\sc Cyril Imbert\par\vspace{3mm}
      CNRS \& Universit\'e de Paris-Est Cr\'eteil\par
      UMR 8050, LAMA\par
      61, avenue du G\'en\'eral de Gaulle 94010 Cr\'eteil,
      France\par\vspace{3mm} e-mail:}
    \tt{cyril.imbert@math.cnrs.fr} \end{center}}

\def\signcm{\bigskip \begin{center} {\sc Cl\'ement
      Mouhot\par\vspace{3mm}
      University of Cambridge\par
      DPMMS, Centre for Mathematical Sciences\par
      Wilberforce road, Cambridge CB3 0WA, UK
      \par\vspace{3mm} e-mail:}
    \tt{C.Mouhot@dpmms.cam.ac.uk} \end{center}}

\begin{document}

\title[H\"older continuity for hypoelliptic
equations]{H\"older continuity of solutions to 
  hypoelliptic equations with bounded measurable coefficients}

\author{C. Imbert \& C. Mouhot} \thanks{The authors would like to
  thank Luis Silvestre for fruitful comments on a preliminary version
  of this article. They also thank Fran\c{c}ois Golse for the idea used
in Lemma~\ref{lem:isoperim} to pass to the limit in the equation.}

\date{\today}


\begin{abstract}
  We prove that $L^2$ weak solutions to  hypoelliptic
  equations with bounded measurable coefficients are H\"older
  continuous. The proof relies on classical techniques developed by De
  Giorgi and Moser together with the averaging lemma and regularity
  transfers developed in kinetic theory. The latter tool is used
  repeatedly: first in the proof of the local gain of integrability of
  sub-solutions; second in proving that the gradient with respect to
  the velocity variable is $L^{2+\eps}_{\mathrm{loc}}$; third, in the
  proof of an ``hypoelliptic isoperimetric De Giorgi lemma''. To get
  such a lemma, we develop a new method which combines the classical
  isoperimetric inequality on the diffusive variable with the
  structure of the integral curves of the first-order part of the
  operator. It also uses that the gradient of solutions w.r.t. $v$ is
  $L^{2+\eps}_{\mathrm{loc}}$.
\end{abstract}

\maketitle

\tableofcontents

\section{Introduction}
\label{sec:introduction}

\subsection{The question studied and its history}

We consider the following nonlinear kinetic Fokker-Planck equation
\begin{equation}
  \label{eq:1}
  \partial_t f + v \cdot \nabla_x f = \rho \, \nabla_v \cdot \left( \nabla_v
    f + v f \right), \quad t \ge 0, \ x \in \R^d, \ v  \in \R^d, 
\end{equation}
(with or without periodicity conditions with respect to the space
variable) where $d \in \N^*$, $f = f(t,x,v) \ge 0$ and $\rho[f] =
\int_{\R^d} f(t,x,v) \dd v$. The construction of global smooth
solutions for such a problem is one motivation for the present paper.

The linear kinetic Fokker-Planck equation
$\partial_t f + v \cdot \nabla_x f = \nabla_v \cdot \left( \nabla_v f
  + v f \right)$
is sometimes called the Kolmogorov-Fokker-Planck equation, as it was
studied by Kolmogorov in the seminal paper \cite{MR1503147}, when
$x \in \R^d$. In this note, Kolmogorov explicitely calculated the
fundamental solution and deduced regularisation in both variables $x$
and $v$, even though the operator
$\nabla_v \cdot (\nabla_v + v) - v \cdot \nabla_x$ shows ellipticity
in the $v$ variable only. It inspired H\"ormander and his theory of
hypoellipticity \cite{MR0222474}, where the regularisation is
recovered by more robust and more geometric commutator estimates (see
also \cite{MR0436223}).

Another question which has attracted a lot of attention in calculus of
variations and partial differential equations along the 20th century
is Hilbert's 19th problem about the analytic regularity of solutions
to certain integral variational problems, when the quasilinear
Euler-Lagrange equations satisfy ellipticity conditions. Several
previous results had established the analyticity conditionally to some
differentiability properties of the solution, but the full answer came
with the landmark works of De Giorgi \cite{MR0082045,MR0093649} and
Nash \cite{MR0100158}, where they prove that  any solution to these
variational problems with square integrable derivative is
analytic. More precisely their key contribution is the
following\footnote{We give the parabolic version due to Nash here.}:
reformulate the quasilinear parabolic problem as 
\begin{equation}
  \label{eq:2}
  \partial_t f = \nabla_v \left( A(v,t) \nabla_v f \right), \quad t \ge 0,
  \ v \in \R^d
\end{equation}
with $f = f(v,t) \ge 0$ and $A = A(v,t)$ satisfies the ellipticity
condition $0 < \lambda I \le A \le \Lambda I$ for two constants
$\lambda,\Lambda >0$ but is, besides that, merely measurable. Then the
solution $f$ is H\"older continuous.  \mk

In view of the nonlinear (quasilinear) equation \eqref{eq:1} it is
natural to ask whether a similar result as the one of De Giorgi-Nash
holds for quasilinear hypoelliptic equations. More precisely, we consider
the following Fokker-Planck equation
\begin{equation}
  \label{eq:main}
  \partial_t f + v \cdot \nabla_x f = \nabla_v \cdot \left(A(x,v,t) \nabla_v
    f \right), \quad t \in (0,T), (x,v) \in \Omega, 
\end{equation}
where $\Omega$ is an open set of $\R^{2d}$, $f = f(t,x,v) \ge 0$ and
the $d \times d$ symmetric matrix $A$ satisfies the ellipticity
condition
\begin{equation}\label{eq:ellipticity}
0 < \lambda I \le A \le \Lambda I
\end{equation} 
for two constants $\lambda,\Lambda$ but is, besides that, merely
measurable. We want to establish the H\"older continuity of $L^2$
solutions to this problem. In order to do so, we first prove that
$L^2$ sub-solutions are locally bounded; we refer to such a result as
an $L^2-L^\infty$ estimate. We then prove that solutions are H\"older
continuous by proving a lemma which is an hypoelliptic counterpart of
De Giorgi's isoperimetric lemma.  
\medskip

Given $z_0= (x_0,v_0,t_0) \in \R^{2d+1}$, $Q =Q_r(z_0)$ denotes a
cylinder centered at $z_0$ of ``radius'' r: it is defined as $Q =
B_{r^3}(x_0) \times B_r(v_0) \times (t_0-R^2,t_0]$ where $B_r(x_0)$
and $B_r(v_0)$ denote the usual Euclidian balls in $x$ and $v$.
\begin{theorem}[H\"older continuity]\label{thm:holder}
  Let $f$ be a solution of \eqref{eq:main} in $Q_0 =Q(z_0,R_0)$ and
  $Q_1 =Q(z_0,R_1)$ with $R_1 < R_0$. Then $f$ is $\alpha$-H\"older
  continuous with respect to $(x,v,t)$ in $Q_1$ and
  \[ \| f \|_{C^\alpha(Q_1)} \le C \|f \|_{L^2(Q_0)}\] for some
  $\alpha$ universal, i.e. $\alpha=\alpha(d,\lambda,\Lambda)$, and $C
  = C(d,\lambda,\Lambda,Q_0,Q_1)$.
\end{theorem}
In \cite{pp}, the authors obtain an $L^2-L^\infty$ estimate with
completely different techniques; however
they cannot reach the H\"older continuity estimate. Our techniques
rely on averaging lemmas \cite{golse,gps} in order to gain some
regularity $H^s_x$, $s>0$ small, in the space variable $x$ from the
natural $H^1_v$ estimate. We emphasize that such $H^s_x$ estimates do
not hold for sub-solutions. From this Sobolev estimate, we can recover
a gain of integrability for $L^2$ sub-solutions, and we then prove the
H\"older continuity through a De Giorgi type argument on the decrease
of oscillation for solutions.

In \cite{wz09,wz11}, the authors get a H\"older estimate for $L^2$
weak solutions of so-called ultraparabolic equations, including
\eqref{eq:main}. Their proof relies on the construction of cut-off
functions and a particular form of weak Poincar\'e inequality
satisfied by non-negative weak sub-solutions.  Our paper proposes a
new, short and simple strategy, that, we hope, sheds new light on the
regularizing effect for hypoelliptic equations with bounded measurable
coefficients and provide tools for further applications. 

We finally mention that Golse and Vasseur proved independently a
similar result \cite{gv}.

\subsection{Plan of the paper}
\label{sec:plan-paper}

In Section~\ref{sec:local-gain-integr}, we first explain how to get a
universal gain of regularity for (signed) $L^2$ solutions; we then
exhibit a universal gain of integrability for non-negative $L^2$
sub-solutions; we finally explain how to derive from this gain of
integrability a local upper bound of such non-negative $L^2$
sub-solutions by using Moser iteration procedure. In
Section~\ref{sec:l2+eps}, we prove that the $v$-gradient of solutions
is $L^{2+\eps}_{\mathrm{loc}}$.  In Section~\ref{sec:osc-decrease},
the H\"older estimate is derived by proving a reduction of oscillation
lemma.

\section{Local gain of regularity / integrability}
\label{sec:local-gain-integr}

We consider the equation \eqref{eq:main} and we want to establish a
local gain of integrability of solutions in order to apply Moser's
iteration and get a local $L^\infty$ bound. Since we will need to
perform convex changes of unknown, it is necessary to obtain this gain
even for (non-negative) \emph{sub-solutions}.

In the two following theorems, we consider cylinders with a scaling
corresponding to the hypoelliptic structure of the equation. For $z_0
= (x_0,v_0,t_0) \in \R^{2d+1}$,
\[ Q_R (z_0)= B_{R^3} (x_0) \times B_R (v_0) \times (t_0-R^2,t_0].\]
The next theorem is stated in cylinders centered at the origin. 
\begin{theorem}[Gain of integrability for non-negative sub-solutions]\label{thm:gain}
  Consider two cylinders $Q_1 =Q_{R_1}(0)$ and $Q_0=Q_{R_0}(0)$ with
  $R_1 < R_0$.  There exists $q>2$ (universal) such that for all
  non-negative $L^2$ sub-solution $f$ of \eqref{eq:main} in $Q_0$, we
  have
\begin{equation}
  \label{eq:gain}
  \n{f}_{L^q (Q_1)} \le C  \n{f}_{L^2(Q_0)}
\end{equation}
where 
\[
C = \bar C\left(\frac1{R_0^2-R_1^2}+\frac{R_0}{R_0^3-R_1^3}+\frac1{(R_0-R_1)^2}\right)
\]
and $\bar C = C(d,\lambda,\Lambda)$.
\end{theorem}
This result is a consequence of the comparison principle and the fact
that, for weak signed solutions $f$, we can even get a gain of
regularity. This gain of regularity will be important in the proof of
the decrease of oscillation lemma to get compactness of sequences of
equi-bounded solutions. This is the reason why it is necessary to
state it in cylinders not necessarily centered at the origin.
\begin{theorem}[Gain of regularity for signed solutions]\label{thm:gain-diff}
  Consider $z_0 \in \R^{2d+1}$ and two cylinders $Q_1 =Q_{R_1}(z_0)$
  and $Q_0=Q_{R_0}(z_0)$ with $R_1 < R_0$.  There exists $s>0$
  (universal) such that for all (signed) $L^2$ weak solution $f$ of
  \eqref{eq:main} in $Q_0$, we have
\begin{equation}
  \label{eq:gain-diff}
  \n{f}_{H^s_{x,v,t} (Q_1)} \le C  \n{f}_{L^2(Q_0)}
\end{equation}
where  and $C =C(d,\lambda,\Lambda,Q_0,Q_1)$.
\end{theorem}

\subsection{Gain of integrability with respect to $v$ and $t$}
\label{subsec:reverse}

The gain of integrability with respect to $v$ and $t$ is classical. It
derives from the natural energy estimate, after truncation.

We follow here \cite{moser} in order to get the following lemma.
\begin{lemma}[Gain of integrability w.r.t. $v$ and $t$]\label{lem:caccio}
Under the assumptions of Theorem~\ref{thm:gain}, the function $f$ satisfies
\begin{eqnarray}\label{eq:h1v-loc}
 \int_{Q_1}  |\nabla_v f|^2 \le C \int_{Q_0} f^2 \\
\nonumber \n{f}^2_{L^2_t L^2_x L^q_v (Q_1)}  \le C  \int_{Q_0} f^2  \\
\nonumber \n{f}^2_{L_t^{\infty} L^2_x L^2_v (Q_1)} \le C\int_{Q_0} f^2
\end{eqnarray}
for some $q>2$ and  $C = \bar C
\left(\frac1{R_0^2-R_1^2}+\frac{R_0}{R_0^3-R_1^3}+\frac1{(R_0-R_1)^2}\right)$
and $\bar C = \bar C (d,\lambda,\Lambda)$.
\end{lemma}
\begin{proof}
Consider $\Psi \in C^\infty_c(\R^{2d} \times \R)$ and integrate the
inequation satisfied by $f$ against $2 f \Psi^2$ in $\R^{2d} \times  [t_1,0] =\cR$
with $t_1 \in (-R_1^2,0]$  and get 
\[ 
\int_{\cR} \partial_t (f^2) \Psi^2 +  \int_{\cR} v \cdot \nabla_x (f^2) \Psi^2
\le 2 \int_{\cR} \nabla_v (A \nabla_v f ) f \Psi^2  .
\]
Add $\int_{\cR} f^2 \partial_t (\Psi^2)$, integrate by parts several times
and use the upper bound on $A$ in order to get 
\begin{multline*}
 \int_{\cR} \partial_t (f^2 \Psi^2) +  2 \int_{\cR} (A \nabla_v f \cdot \nabla_v f)
\Psi^2 \\
 \le  \int_{\cR} f^2 (\partial_t 
+ v \cdot \nabla_x )(\Psi^2) + 2 \int_{\cR}  \Psi \sqrt{A} \nabla_v f \cdot  f \sqrt{A} \nabla_v \Psi \\
 \le  \int_{\cR} f^2 (\partial_t + v \cdot \nabla_x )(\Psi^2) 
+  \int_{\cR}  (A \nabla_v f \cdot \nabla_v f) \Psi^2 +  \int_{\cR} f^2 (A \nabla_v \Psi \cdot \nabla_v \Psi). 
\end{multline*}
We thus get 
\begin{multline*}
  \int_{\cR} \partial_t (f^2 \Psi^2) + \lambda \int_{\cR} |\nabla_v f|^2
  \Psi^2  \\ 
\le \bar{C} \left( \n{\partial_t \Psi}_{\infty} + R_0 \n{\nabla_x \Psi}_{\infty} +
    \n{\nabla_v \Psi}_{\infty}^2 \right)   \int_{\cR \cap \supp \Psi} f^2 
\end{multline*}
with $\bar{C} = C(\Lambda,d)$.  Choose next $\Psi^2$ such
that $\Psi (t=0)=0$ and $\supp \Psi \subset Q_0$ and get
\[ \int_{x,v} f^2  \Psi^2 (t_1) + \lambda \int |\nabla_v f|^2 \Psi^2
\le  C C_{0,1}  \int_{Q_0} f^2. \]

If $\Psi^2$ additionally satisfies $\Psi^2 \equiv 1$ in $Q_1$, we get
\eqref{eq:h1v-loc}. The Sobolev inequality  then implies the estimate
for $\n{f}^2_{L^2_t L^2_x L^q_v (Q_1)}$. If now $t_1 \in [t_0-r_1^2,t_0]$ is arbitrary, 
we get the  estimate for
$\n{f}^2_{L_t^{\infty} L^2_x L^2_v (Q_1)}$. The proof is now complete. 
\end{proof}

\subsection{Gain of regularity with respect to $x$ for signed weak solutions}

\begin{lemma}[Gain of regularity w.r.t. $x$]\label{lem:gain-x-sol}
  Under the assumptions of Theorem~\ref{thm:gain}, if $f$ is a signed
  weak solution to \eqref{eq:main}, 
\begin{equation}
  \label{eq:gain-x-sol}
  \n{D^{1/3}_x f}_{L^2(Q_1)} \le  C  \n{f}_{L^2
    (Q_0)}
\end{equation}
with $C = \bar C \left(\frac1{R_0^2 -R_1^2}+\frac{R_0}{R_0^3-R_1^3}+ \frac1{(R_0-R_1)^2} 
\right)$ and $\bar C= \bar C (d,\lambda,\Lambda)$. 
In the case $q>1$ with $\cQ_1$ instead of $Q_1$, we have
\begin{equation}
  \label{eq:gain-x-sol-q}
  \n{D^{1/3}_x f}_{L^q(\cQ_1)} \le  C  \n{\nabla_v f}_{L^q
    (Q_0)}
\end{equation}
with $C = C (d,\lambda,\Lambda,Q_0,\cQ_1)$. 
\end{lemma}
\begin{proof}
  Let $R_{\frac12}=\frac{R_1+R_0}2$ and $Q_{\frac12} =Q_{R_{\frac12}}$. In particular, 
\[
Q_1 \subset Q_{\frac12}  \subset Q_0.
\]  
For $i=1,\frac{1}2$, consider $f_i = f \chi_i$ where $\chi_1$ and
$\chi_{\frac12}$ are two truncation functions such that
\begin{align*}
 \chi_1 \equiv 1 \text{ in } Q_1  & \quad \text{ and } \quad 
\chi_1 \equiv 0 \text{ outside } Q_{\frac12} \\
 \chi_{\frac12} \equiv 1 \text{ in } Q_{\frac12} & \quad \text{ and } \quad 
\chi_{\frac12} \equiv 0 \text{ outside } Q_0.
\end{align*}  
We get
\[   (\partial_t  + v \cdot \nabla_x) f_1 = \nabla_v \cdot H_1 + H_0
\text{ in } \R^{2d} \times (-\infty;0] \]
\[
\text{with } \quad \begin{cases}
H_1 &= \chi_1 A \nabla_v f_{\frac12} \\
H_0 &= -\nabla_v \chi_1 \cdot A \nabla_v f_{\frac12} + \alpha_1 f_\frac12 \\
\alpha_1 &=  (\partial_t + v \cdot \nabla_x ) \chi_1.
\end{cases}
\]
The previous equation holds true in $\R^{2d} \times (-\infty;0]$ since
$f_1$, $H_0$ and  $H_1$  are supported in $Q_0$. 
We remark that using \eqref{eq:h1v-loc},
\[ 
\|H_0\|_{L^2}+ \|H_1 \|_{L^2} \le C \|f\|_{L^2(Q_0)}
\] 
with $C$ as in the statement.  Applying \cite[Theorem~1.3]{bouchut}
with $p=2$, $r=0$, $\beta =1$, $m=1$, $\kappa=1$ and $\Omega =0$
yields \eqref{eq:gain-x-sol}.  To get \eqref{eq:gain-x-sol-q}, we
simply use a cut-off function such that $\alpha_1 \equiv 0$ and we
apply \cite[Theorem~1.3]{bouchut} with $p=q$, $r=0$, $\beta =1$,
$m=1$, $\kappa=1$ and $\Omega =0$. The proof is now complete.
\end{proof}

\subsection{Gain of integrability with respect to $x$ for non-negative sub-solutions}

\begin{lemma}[Gain of integrability w.r.t. $x$]\label{lem:gain-x}
Under the assumptions of Theorem~\ref{thm:gain}, there exists $p>2$
such that 
\begin{equation}
  \label{eq:gain-x}
  \n{f}_{L^2_t L^p_x L^1_v (Q_1)} \le C   \n{f}_{L^2
    (Q_0)}
\end{equation}
with $C  = \bar C\left(\frac1{R_1-R_0}+ \frac1{(r_1-r_0)^2} + \frac1{\tau_1-\tau_0}
\right)$ and $\bar C=\bar C(d,\lambda,\Lambda)$.
\end{lemma}
\begin{proof}
We follow the reasoning of Lemma~\ref{lem:gain-x-sol}. The function $f_1$ 
now satisfies the following inequation
\[   (\partial_t  + v \cdot \nabla_x) f_1 \le \nabla_v \cdot H_1 + H_0
\text{ in } \R^{2d} \times \R. \]
If $g$ solves 
\[
\begin{cases}   (\partial_t  + v \cdot \nabla_x) g = \nabla_v \cdot H_1 + H_0\\
  g(x,v,-R_1^2) = f_1 (x,v,-R_1^2)
\end{cases} 
\] 
then the comparison principle implies that $f_1 \le g$ in $\R^{2d}
\times [-R_1^2,0]$. Applying Lemma~\ref{lem:gain-x-sol} and
Sobolev inequality,  we get
\[
\|f_1 \|_{L^2_t L^p_x L^1_v}  \le \|g \|_{L^2_t L^p_x L^1_v}  \le C C_{0,1}  \|f_0\|_2 
\]
(where $p = 2d/(d-2/3) >2$) which yields the desired estimate.
\end{proof}

\subsection{Proof of Theorems~\ref{thm:gain} and \ref{thm:gain-diff}}

\begin{proof}[Proof of Theorem~\ref{thm:gain}]
Combine Lemmas~\ref{lem:caccio} and \ref{lem:gain-x}, use
interpolation to get the result through a covering argument. 
\end{proof}

\begin{proof}[Proof of Theorem~\ref{thm:gain-diff}]
  We first prove the result when cylinders are centered at the
  origin. In this case, it is enough to combine
  Lemmas~\ref{lem:caccio} and \ref{lem:gain-x-sol}, Aubin's lemma and
  use interpolation to get the result.

For cylinders that are not centered at the origin, we use \emph{slanted} cylinders 
of the form: 
\[ \tilde{Q}_R (z_0) = \{ (x,v,t): |x-x_0 - (t-t_0)v_0| < R^3,
|v-v_0|< R, t \in (t_0-R^2,t_0]\}.\] Now we cover $Q_0$ and $Q_1$ with
such slanted cylinders, we get the gain of regularity (whose exponent
remains universal) and we get the desired result. 
\end{proof}

\subsection{Local upper bounds for non-negative sub-solutions}
\label{sec:supr-infim-bounds}

In this subsection, we iterate the local gain of integrability to
prove that non-negative $L^2$ sub-solutions are in fact locally
bounded (with an estimate). 
\begin{theorem}[Upper bounds for non-negative $L^2$
  sub-solutions] \label{thm:sup-sub} Given two cylinders
  $Q_0=Q_{R_0}(z_0)$ and $Q_\infty=Q_{R_\infty} (z_0)$, let $f$ be a
  non-negative $L^2$ sub-solution of
\[ (\partial_t + v\nabla_x) f \le \nabla_v (A \nabla_v f) \quad \text{ in } Q_0.\]
Then
\[ \sup_{Q_\infty} f \le C \| f\|_{L^2(Q_0)}\]
for some $C=C(d,\lambda,\Lambda,Q_0,Q_{\infty})$. 
\end{theorem}
\begin{proof}
We first prove the result for cylinders centered at the origin. 
  To do so, we first remark that, for all $q>1$, the function $f^q$
  satisfies
\[ (\partial_t + v\nabla_x) f^q \le \nabla_v \cdot (A \nabla_v f^q) \quad
\text{ in } Q_0 .\]

We now rewrite \eqref{eq:gain} from $Q_q =Q_{R_q} (z_0)$ to $Q_{q+1}$ with $R_{q+1} < R_q$ 
as follows: 
\begin{equation}\label{eq:gain-again}
  \n{(f^q)^\kappa}_{L^2 (Q_{q+1})}^2 \le C_{q+1}  \n{f^q}_{L^2(Q_q)}^{2\kappa} 
\end{equation}
where $\kappa = p /2 > 1$ and 
\[ C_{q+1} = \bar C \left[\frac1{R_q^2-R_{q+1}^2}+\frac{R_q}{R_q^3-R_{q+1}^3}+\frac1{(R_q-R_{q+1})^2}\right]^\kappa \]
with $\bar C = \bar C (d,\lambda,\Lambda)$. 

Choose now $q = q_n = 2 \kappa^n$ for $n \in \N$, simply write
$Q_n$ for $Q_{q_n}$ and $C_n$ for $C_{q_n}$ and get from
\eqref{eq:gain-again} 
\begin{equation}\label{eq:moser}
 \n{f^{q_{n+1}}}^2_{L^2(Q_{n+1})} \le  C_{n+1} \n{f^{q_n}}^{2\kappa}_{L^2(Q_n)}.  
\end{equation}
Moreover, we choose 
\[R_{n+1} = R_n - \frac{1}{a(n+1)^2}\]
for some $a>0$ so that 
\[ C_n \sim \bar C (a^2 n^4 + b n^2)^\kappa \]
with $b= \frac{5a}{6R_\infty}$. 
Applying iteratively \eqref{eq:moser}, we get the result if 
\[ \prod_{n=0}^{+\infty} C_n^{\frac1{2\kappa^n}} < +\infty\] which
indeed holds true. This yields the desired result in the case of
cylinders centered at the origin.

For cylinders that are not centered at the origin, we argue as in the
proof of Theorem~\ref{thm:gain-diff}. The proof is now complete.
\end{proof}

\section{Gain of integrability for the gradient  w.r.t. the velocity variable}
\label{sec:l2+eps}

This subsection is devoted to the proof of the following theorem. 
\begin{theorem}[Gain of integrability for $\nabla_v f$]\label{thm:l2+eps}
  Let $f$ be a solution of \eqref{eq:main} in some cylinder $Q_0
  =Q_{R_0} (z_0)$. There exists a universal $\eps >0$ 
  such that for all $Q_i = Q_{R_i}(z_0)$, $i=1,2$ with $R_2<R_1 < R_0$, 
 $\nabla_v f \in L^{2+\eps}(Q_2)$
\begin{equation}\label{eq:l2+eps}
\int_{Q_2} |\nabla_v f|^{2+\eps} dz \le C \left( \int_{Q_1} |\nabla_v f |^2 dz \right)^{\frac{2+\eps}2}
\end{equation}
with $C = C(d,\lambda,\Lambda,Q_2,Q_1,Q_0)$. 
\end{theorem}
The proof follows along the lines of the one of
\cite[Theorem~2.1]{gs}. It consists in deriving a reverse H\"older
inequality which in turn implies the result thanks to the
analogous of \cite[Proposition~1.3]{gs}.
\begin{lemma}[A Gehring lemma]\label{lem:gehring}
  Let $g \ge 0$ in $Q$ such that there exists $q >1$ such that for all
  $z_0 \in Q$ and $R$ such that $Q_{4R} (z_0) \subset Q$, 
\[ \fint_{Q_R(z_0)} g^q \,dz \le b \left(\fint_{Q_{8R}(z_0)} g \, dz \right)^q + \theta \fint_{Q_{8R}(z_0)} g^q \,dz \]
for some $\theta >0$. There exists $\theta_0 = \theta_0 (q,d)$ such that 
if $\theta < \theta_0$, then $g \in L^p_{\mathrm{loc}} (Q)$ for $p \in [q,q+\eps)$ and 
\[ \left( \fint_{Q_R} g^p \,dz \right)^{\frac1p} \le c \left( \fint_{Q_{4R}} g^q \,dz \right)^{\frac1q}, \]
the constants $c$ and $\eps>0$ depending only on $b,q,\theta$ and dimension. 
\end{lemma}
The proof of Lemma~\ref{lem:gehring} is an easy adaptation of the one
of \cite[Proposition~5.1]{gm}, by changing Euclidian cubes with
cylinders $Q_R$. \bigskip

The proof of Theorem~\ref{thm:l2+eps} is a consequence of some
estimates involving weighted means of the solution. Given $z_0 \in
\R^{2d+1}$, they are defined as follows of $f$ are defined as follows:
\[ \tilde{f}_{2R} (t) = (cR^{4d})^{-1} \int f(t,x,v) \chi_{2R} (x,v,t) dx dv   \]
(for some $c$ defined below) where $\chi_{2R}$ is a cut-off function such that 
\[ \chi_{2R} (x,v,t) = \phi_{R^3}((x-x_0)-(t-t_0)(v-v_0)) \phi_R(v-v_0) \] with $\phi_R (a) =
\phi(a/R)$ for some $\phi$ such that $\sqrt{\phi} \in C^\infty(\R^d)$
and $\phi \equiv 1$ in $B_1$ and $\supp \phi\subset B_2$.  In
particular,
\[ (\partial_t + v \cdot \nabla_x ) \chi_R = 0 \quad \text{ and } \int
\chi_R (x,v,t) \, dx dv = \int \phi_{R^3} \int \phi_R = cR^{4d}\] with $c = (\int
\phi)^2$. We now introduce ``sheared'' cylinders $\cQ_R (z_0) = z_0 + \cQ_R$ with
\[ \cQ_R = \{ (x,v,t): |x-tv| < R^3, |v| < R, t \in (-R^2,0] \}. \]
Remark that 
\begin{equation}\label{eq:cyl}
Q_{2^{-1/3} R} \subset \cQ_R \subset Q_{2^{1/3} R}.
\end{equation}
Remark also that
$\chi_{2R} \equiv 1$ in $\cQ_R$ and $\chi_{2R} \equiv 0$ outside $\cQ_{2R}$. 
\begin{lemma}[Estimates] \label{lem:caccio-bis} Let $f$ be a solution
  of \eqref{eq:main} in $\cQ_0$. Then for $Q_{3R} (z_0) \subset
  \cQ_0$,
\begin{eqnarray}
\label{eq:caccio-mean}
 \int_{\cQ_R(z_0)} |\nabla_v f |^2 \, dz \le C R^{-2} \int_{\cQ_{2R}(z_0)} |f - \tilde{f}_{2R}|^2 \, dz \\
\label{eq:poincare}
 \sup_{t \in (t_0-R^2,t_0]} \int_{\cQ_R^t(z_0)} |f(t) - \tilde{f}_R (t)|^2  \le C R^2 \int_{\cQ_{3R}(z_0)} |\nabla_v f|^2 \, dz
\end{eqnarray}
where $\cQ_R^t (z_0)= z_0 + \{ (x,v): |x-tv| < R^3, |v| < R \}$.
\end{lemma}
\begin{remark}
This lemma corresponds to \cite[Lemmas~2.1 \& 2.2]{gs}. 
\end{remark}
\begin{proof}
For the sake of clarity, we put $z_0=0$. 
  Consider $\tau_{2R} \in C^\infty (\R,\R)$ such that $0 \le \tau_{2R} \le 1$,
  $\tau_{2R} \equiv 0$ in $(-\infty,-(2R)^2]$ and $\tau_{2R} \equiv 1$ in
  $[-R^2,0]$.  Use $2(f-\tilde{f}_{2R}) \chi_{2R}  \tau_{2R} $ as a test
  function for \eqref{eq:main} and get 
\begin{multline*}
\int  (f(0)-\tilde{f}_{2R}(0))^2 \chi_{2R} \, dx dv + 2 \int (A \nabla_v f \cdot \nabla_v f) \chi_{2R} \tau_{2R} \, dz \\
= \int(f-\tilde{f}_{2R})^2 \chi_{2R} (\partial_t \tau_{2R} ) 
-   \int v \cdot \nabla_x \left[ (f - \tilde{f}_{2R} )^2 \right] \chi_{2R} \tau_{2R} \\
- 2 \int (f-\tilde{f}_{2R}) A \nabla_v f \cdot \nabla_v \chi_{2R} \tau_{2R} .
\end{multline*}
Remark that the definition of $\tilde{f}_{2R}$ implies that the remaining
term \[- 2 \int (\partial_t \tilde{f}_{2R}) (f-\tilde{f}_{2R}) \chi_{2R} \tau_{2R}\]
vanishes. This equality yields 
\begin{align*}
\int  (f(0)-\tilde{f}_{2R}(0))^2 \chi_{2R} \, dx dv + \lambda \int |\nabla_v f|^2 \chi_{2R} \tau_{2R} \, dz \\
\le \int (f-\tilde{f}_{2R})^2 \left(\chi_{2R} |\partial_t \tau_{2R} | 
+ |v \cdot \nabla_x \chi_{2R}| \tau_{2R}  + \frac{\Lambda^2}{\lambda} | \nabla_v \sqrt{\chi_{2R}} |^2 \tau_{2R} \right)
\end{align*}
which yields \eqref{eq:caccio-mean}. Changing the final time, we also get
\[
\sup_{t \in (-R^2,0]} \int  (f(t)-\tilde{f}_{2R}(t))^2 \chi_{2R} (t) \, dx dv 
\le C  R^{-2} \int_{\cQ_{2R}} |f - \tilde{f}_{2R}|^2 \, dz.
\]
Now the function $F= f -\tilde{f}_{2R}$ is such that $\int F(x,v,t) dx dv =0$. 
In particular, we have 
\[ 
\int_{\cQ_{2R}} (f-\tilde{f}_{2R})^2 \, dz \le C \int_{\cQ_{2R}} 
(R^2 |\nabla_v f|^2 + R^{2s} |D^s_x f|^2) \, dx dv dt. 
\] 
Arguing as in the proof of Lemma~\ref{lem:gain-x-sol} with a cut-off
function $\chi_1 = \chi_{2R}$ which satisfies $(\partial_t + v \cdot
\nabla_x) \chi_1 =0$, we get
\[ \int_{\cQ_{2R}} R^{2s} |D^s_x f|^2 \, dx dv dt \le C \int_{\cQ_{3R}} 
R^2 |\nabla_v f|^2  \, dx dv dt.\]
Combining the three previous estimates yields 
\[ 
\sup_{t \in ( -R^2,0]} \int  (f(t)-\tilde{f}_{2R}(t))^2 \chi_{2R}(t) \, dx dv 
\le C R^2 \int_{\cQ_{3R}}  |\nabla_v f|^2  \, dx dv dt. 
\]
Finally, we write for $t \in (-R^2,0]$ 
\begin{align*}
\frac12 \int_{\cQ_R^t}  (f(t)-\tilde{f}_{R}(t))^2 \chi_{2R}(t) \le & \int_{\cQ_R^t}  (f(t)-\tilde{f}_{2R}(t))^2 \chi_{2R}(t)\\
& + \int_{\cQ_R^t}  (\tilde{f}_{2R}(t)-\tilde{f}_{R}(t))^2 \chi_{2R}(t)\\
\le &  \int  (f(t)-\tilde{f}_{2R}(t))^2 \chi_{2R}(t) \\
& + |\cQ_R^t| \left( (cR^{4d})^{-1} \int (f- \tilde f_{2R} (t)) \chi_{R} (x,v,t) \, dx dv \right)^2 \\
& \le C \int_{\cQ_R^t}  (f(t)-\tilde{f}_{2R}(t))^2 \chi_{2R}(t) 
\end{align*}
and we get the second desired estimate since $\chi_{2R} \equiv 1$ in $\cQ_R$. 
\end{proof}
We now turn to the proof of Theorem~\ref{thm:l2+eps}. The use of
\eqref{eq:gain-x-sol-q} is the main difference with \cite{gs}.
\begin{proof}[Proof of Theorem~\ref{thm:l2+eps}]
  Pick $p>2$ and let $q$ denotes its conjugate exponent: $\frac1q +
  \frac1{p} =1$. We follow \cite{gs} in writing (omitting the center of cylinders $z_0$),
\begin{align*}
\int_{\cQ_{2R}} |f - \tilde{f}_{2R}|^2 \,   \\
\le \sup_{t \in (t_0-(2R)^2,t_0]} \left(\int_{\cQ^t_{2R}} |f - \tilde{f}_{2R}|^2 \right)^{\frac12}  
\int_{t_0-(2R)^2}^{t_0} dt \left( \int_{\cQ^t_{2R}} |f - \tilde{f}_{2R}|^2\right)^{\frac12} \\
\lesssim  R \left( \int_{\cQ_{4 R}} |\nabla_v f|^2 \,  \right)^{\frac12} 
 \int_{t_0-(2R)^2}^{t_0} dt  \left( \int_{\cQ^t_{2R}} |f - \tilde{f}_{2R}|^q\right)^{\frac1{2q}} \left( \int_{\cQ^t_{2R}} |f - \tilde{f}_{2R}|^{p}\right)^{\frac1{2p}}
\end{align*}
where   \eqref{eq:poincare} and H\"older inequality are used successively. 

We now use Sobolev inequalities and H\"older inequality (twice)  successively to  get 
\begin{align*}
\int_{\cQ_{2R}} |f - \tilde{f}_{2R}|^2 \,  \lesssim  
R \left( \int_{\cQ_{4 R}} |\nabla_v f|^2 \,  \right)^{\frac12} \\
\times  \int_{t_0-(2R)^2}^{t_0} dt  \left( \int_{\cQ^t_{2R}} R^q |\nabla_v f|^q + R^{q/3} |D_x^{1/3} f|^q \right)^{\frac1{2q}} 
  \left( \int_{\cQ^t_{2R}}  R^2 |\nabla_v f |^2 + R^{2/3} |D_x^{1/3} f|^2 \right)^{\frac1{4}} \\
 \lesssim R \left( \int_{\cQ_{4 R}} |\nabla_v f|^2 \,  \right)^{\frac12} \\
\times \left( \int_{\cQ_{2R}} R^q |\nabla_v f|^q + R^{q/3} |D_x^{1/3} f|^q \right)^{\frac1{2q}} 
 \left( \int_{t_0-(2R)^2}^{t_0} \left(\int_{\cQ^t_{2R}}  R^2 |\nabla_v f |^2 + R^{2/3} |D_x^{1/3} f|^2 \right)^{\frac{q}{2(2q-1)}} \right)^{\frac{2q-1}{2q}} \\
 \lesssim R \left( \int_{\cQ_{4 R}} |\nabla_v f|^2 \,  \right)^{\frac12} \\
\times \left( \int_{\cQ_{2R}} R^q |\nabla_v f|^q + R^{q/3} |D_x^{1/3} f|^q \right)^{\frac1{2q}} 
 \left( \int_{\cQ_{2R}}  R^2 |\nabla_v f |^2 + R^{2/3} |D_x^{1/3} f|^2  \right)^{\frac14} R^{\frac32q -1}.
\end{align*}
We now use \eqref{eq:gain-x-sol-q} and get 
\begin{align*}
\int_{\cQ_{2R}} |f - \tilde{f}_{2R}|^2 \,  \lesssim  R^{\frac32q+1}
\left( \int_{\cQ_{4 R}} |\nabla_v f|^2 \,  \right)^{\frac12} 
 \left( \int_{\cQ_{2R}}  |\nabla_v f|^q \right)^{\frac1{2q}} 
 \left( \int_{\cQ_{2R}}   |\nabla_v f |^2 \right)^{\frac14}  \\
\lesssim R^{\frac32q+1} \left( \int_{\cQ_{4 R}} |\nabla_v f|^2 \,  \right)^{\frac34}  \left( \int_{\cQ_{2R}}  |\nabla_v f|^q \right)^{\frac1{2q}}.
\end{align*}
Now use \eqref{eq:caccio-mean} and get for all $\eps>0$,
\begin{align*}
  \fint_{\cQ_{R}} |\nabla_v f |^2   & \lesssim R^{\frac32q-1} |\cQ_{2R}|^{\frac1{2q} - \frac14}
  \left( \fint_{\cQ_{4 R}} |\nabla_v f|^2 \, 
  \right)^{\frac34} \left( \fint_{\cQ_{4R}} |\nabla_v f|^q 
  \right)^{\frac1{2q}} \\
& \lesssim R^{\gamma_d}
  \left( \fint_{\cQ_{4 R}} |\nabla_v f|^2  
  \right)^{\frac34} \left( \fint_{\cQ_{4R}} |\nabla_v f|^q 
  \right)^{\frac1{2q}} \\
& \lesssim \eps \fint_{\cQ_{4 R}} |\nabla_v f|^2 + c_{\eps} R^{4\gamma_d} \left( \fint_{\cQ_{4R}} |\nabla_v f|^q 
  \right)^{\frac2{q}}
\end{align*}
where $\gamma_d = (4d+2)(\frac1{2q}-\frac14) + \frac32 q-1>0$. Using \eqref{eq:cyl}, we finally get  
\[ 
\fint_{Q_{R}} |\nabla_v f |^2 \lesssim \eps \fint_{Q_{8 R}}
|\nabla_v f|^2 + c_{\eps} R^{4\gamma_d} \left( \fint_{Q_{8R}}
  |\nabla_v f|^q \right)^{\frac2{q}}.
\] 
Apply now Proposition~\ref{lem:gehring} in order to achieve the proof
of Theorem~\ref{thm:l2+eps}.
\end{proof}

\section{The decrease of oscillation lemma}
\label{sec:osc-decrease}

It is classical that H\"older continuity is a consequence of the
decrease of the oscillation of the solution ``at unit scale''. 
\begin{lemma}[Decrease of oscillation]\label{lem:osc-decrease}
Let $f$ be a solution of \eqref{eq:main} in $Q_2 =  B_2(x_0)
  \times B_2(v_0) \times (-2,0)$ with $|f|\le 1$. 
Then 
\[ \osc_{Q_{\frac12}} f \le 2 - \lambda \] with
$Q_{\frac12}=B_{\frac12}(x_0) \times B_{\frac12}(v_0) \times
(-\frac12,0)$ for some $\lambda \in (0,2)$ only depending on dimension
and ellipticity constants.
\end{lemma}
\begin{remark}\label{rem:scale}
 The equation is ``invariant'' under the following scaling 
\[(x,v,t) \mapsto (r^{-3} x, r^{-1} v, r^{-2} t);\] 
indeed, it changes
$A(x,v,t)$ into $A (r^{-3} x,r^{-1}v, r^{-2} t)$ which still satisfies
\eqref{eq:ellipticity}.  
\end{remark}
This lemma is an immediate consequence of the following one. 
\begin{lemma}[Decrease of the supremum bound]\label{lem:sup-decrease}
  Let $f$ be a solution of \eqref{eq:main} in $Q_2$ with $|f|\le 1$.  If
\[ |\{ f \le 0 \} \cap Q_1 | \ge \frac12 |Q_1| \]
with $Q_1 =  B_1 (x_0) \times B_1(v_0) \times (-1,0)$, then 
\[ \sup_{Q_{\frac12}} f \le 1 - \lambda \] for some $\lambda \in
(0,2)$ only depending on dimension and ellipticity constants.
\end{lemma}
As explained in \cite{vasseur} for instance, this lemma itself is a
consequence of the following one. The details are given in Appendix
for the reader's convenience.
\begin{lemma}[A De Giorgi-type lemma]\label{lem:isoperim}
For all $\delta_1>0$ and $\delta_2>0$, there exists $\alpha >0$ such that for all
solution $f$ of \eqref{eq:main} in $Q_2$ with $|f| \le 1$ and
\[\begin{aligned}
|\{ f \ge \frac12 \} \cap Q_1 | &\ge  \delta_1 \\
|\{ f \le 0 \} \cap Q_1 | & \ge \delta_2
\end{aligned}
\]
we have
\[ |\{ 0 < f < \frac12\} \cap Q_1| \ge \alpha.\]
\end{lemma}
\begin{remark}
  It is important to emphasize that the lemma is stated for solutions
  of \eqref{eq:main}, not sub-solutions. 
\end{remark}
\begin{remark}
  The idea of proving such a generalization of the classical
  isoperimetric lemma of De Giorgi is reminiscent of an argument of
  Guo \cite{guo}.  See also the very nice survey by Vasseur
  \cite{vasseur}.
\end{remark}
\begin{proof}
  We argue by contradiction by assuming that there exists a sequence
  $f_k$ of solutions of \eqref{eq:main} for some diffusion matrix
  $A_k$ such that $|f_k| \le 1$ and
\[
\begin{aligned}
|\{ f_k \ge \frac12 \} \cap Q_1 | &\ge  \delta_1 \\
|\{ f_k \le 0 \} \cap Q_1 | & \ge \delta_2 \\
|\{ 0 < f_k < \frac12\} \cap Q_1| & \to 0 \quad \text{ as } k \to +\infty. 
\end{aligned}
\]

\smallskip

\noindent
\textbf{Compactness in $L^2$.}
Since the sequence $f_k$ is bounded in $L^2(Q_2)$,
Theorem~\ref{thm:gain} implies that it is relatively compact in
$L^2(Q_1)$ for any $Q_1 \Subset Q_2$. With thus can assume that
$f_k$ converges in $L^2(Q_1)$ towards $f$ as $k \to
+\infty$. In particular,  it satisfies
\begin{eqnarray}
\nonumber |\{ f \ge \frac12 \} \cap Q_1 | &\ge  \frac{\delta_1}2 \\
\nonumber |\{ f \le 0 \} \cap Q_1 | & \ge \frac{\delta_2}2 \\
\label{eq:nulset} |\{ 0 < f < \frac12\} \cap Q_1| & =0.
\end{eqnarray}
Moreover, the natural energy estimate for solutions of \eqref{eq:main}
implies that $f \in L^2_{t,x} H^1_v$ by weak limit. Hence, by the
classical de Giorgi isoperimetric inequality, for almost every
$(t,x) \in B_1(x_0) \times (-1,0)$, we have
\[ \left\{
\begin{aligned}
 \text{either } \quad  &  \mbox{for almost every }  v \in B_1(v_0), \quad  f(t,x,v) \le 0 \\
\text{ or } \quad   \quad & \mbox{for almost every } v\in B_1 (v_0), \quad  f(t,x,v) \ge \frac12.
\end{aligned}
\right.
\]
\smallskip

\noindent
\textbf{Truncation.}  Consider now a smooth non-decreasing function
$T: [-1,1] \to \R$ such that $T \equiv 0$ in $[-1,0]$ and
$T \equiv \frac12$ in $[\frac12,1]$.  We have that
$\bar{f}_k = T (f_k)$ satisfies $\bar{f}_k \to \bar f$ in $L^2(Q_1)$
such that
\[ 
\left\{
\begin{aligned}
 \text{either } \quad  &  \mbox{for almost every }  v \in B_1(v_0), \quad \bar f(t,x,v) = 0 \\
\text{ or } \quad   \quad & \mbox{for almost every } v\in B_1 (v_0), \quad \bar f(t,x,v) = \frac12.
\end{aligned}
\right.
\]
In particular, 
\[ 
\nabla_v \bar f = 0 \text{ in } L^2(Q_1)
\] 
i.e. the function is everywhere a \emph{local equilibrium} in the
terminology of kinetic theory. Hence,
\[ \bar f (t,x,v) = \bar f (t,x) \in \{0,\frac12\} \]
and 
\begin{equation}
\label{eq:measures}
\left\{\begin{aligned}
| \{ \bar f = \frac12 \} \cap  B_1 \times (-1,0)| \ge \frac{\delta_1}{|B_1|} \\
| \{ \bar f = 0 \} \cap B_1 \times (-1,0) | \ge \frac{\delta_2}{|B_1|} 
\end{aligned}\right.
\end{equation}
 \smallskip

\noindent
\textbf{Passage to the limit.}
The function $\bar{f}_k$ satisfies in $Q_1$,
\begin{equation}\label{eq:k}
\partial_t \bar{f}_k + v \cdot \nabla_x \bar{f}_k = \nabla_v \cdot
(A_k \nabla_v \bar{f}_k) - T''(f_k) A_k
\nabla_v f_k \cdot \nabla_v f_k.
\end{equation}
For a test function $\phi$ supported in $Q_1$, we can write
\[ 
\left| \int T''(f_k) A_k \nabla_v f_k \cdot \nabla_v \tilde{f}_k \phi \right| \le \Lambda
\|T''\|_\infty \|\phi\|_\infty \int_{B_k} |\nabla_v f_k|^2
\]
where 
\[ B_k = \{ 0 < \tilde{f}_k < \frac12 \} \cap Q_1.\]
In view of \eqref{eq:nulset}, we know that $|B_k|\to 0$ as
$k \to +\infty$.  In view of Theorem~\ref{thm:l2+eps}, this implies
that
\begin{equation}\label{eq:k2}
  \int T''(f_k) A_k \nabla_v f_k \cdot \nabla_v f_k \phi \to 0 \quad \text{ as } n \to
  +\infty.
\end{equation}
We also know that $\nabla \bar{f}_k$ is bounded in $L^2(Q_1)$. Hence,
we can assume that 
\begin{equation} \label{eq:k3}
 \bar h_k := A_k \nabla_v \bar f_k \rightharpoonup \bar h \quad 
\text{ in } L^2(Q_1).
\end{equation}
In view of \eqref{eq:k}, \eqref{eq:k2} and \eqref{eq:k3}, we thus have
\begin{equation}\label{eq:barf}
 (\partial_t + v \cdot \nabla_x) \bar f = \nabla_v \bar h.
\end{equation}
\smallskip

\noindent
\textbf{Identification of $\bar h$.}  Given $\phi \in \mathcal{D}
(Q_1)$, we can on one hand use $\bar f \phi$ as a test function in
\eqref{eq:barf} and get after integrating in all variables,
\[ \frac12 \int (\bar f)^2 (\partial_t + v\cdot \nabla_x) \phi = \int 
\bar h \nabla_v (\bar f \phi).
\]
On the other hand, we can use $\bar f_k \phi$ as a test function in 
\eqref{eq:k} and get at the limit
\[ \frac12 \int (\bar f)^2 (\partial_t + v\cdot \nabla_x) \phi = \lim_{k\to+\infty} \int \bar h_k \cdot 
\nabla_v (\bar f_k \phi).\]
In particular,
\[\int \bar h \nabla_v (\bar f \phi)= \lim_{k\to+\infty} \int \bar h_k \cdot 
\nabla_v (\bar f_k \phi).\]
Since $f_k \to f$ strongly in $L^2$ we have 
\[
 \lim_{k\to+\infty} \int \bar h_k \cdot \bar f_k \nabla_v \phi =
 \int \bar h \cdot \bar f \nabla_v \phi.
\]
and then since $\nabla_v \bar f=0$, this implies
\[ \lim_{k\to+\infty} \int \bar h_k \cdot \nabla_v \bar f_k \phi =0.\]
Hence, for $\phi \ge 0$, 
\begin{align*}
\int |\bar h|^2 \phi & \le \liminf_{k \to +\infty} \int |\tilde{A}_k \nabla_v \bar f_k|^2 \phi \\
& \le \Lambda \lim_{k \to +\infty} \int \tilde{A}_k \nabla_v \bar f_k \cdot \nabla_v \bar f_k \phi \\
& \le \Lambda \lim_{k \to +\infty} \int \bar h_k \cdot \nabla_v \bar f_k \phi=0.
\end{align*}
which implies that $\bar h=0$. 
\smallskip

\noindent
\textbf{Conclusion.} We deduce that 
\begin{align*}
\text{for a.e. } v \in B_1(0), \quad  \partial_t \bar f + (v_0+v) \cdot \nabla_x \bar f = 0 \quad \text{ in }  B_1 \times (-1,0).
\end{align*}
In particular, rewriting the equation for $-v$, summing and using all $v \in B_1 (0)$, we get 
\[ \partial_t f + v_0 \cdot \nabla_x f \equiv 0, \nabla_x f \equiv
0 \]
which, in turn, yields that $f$ is constant (i.e. is a \emph{global
  equilibrium} in the terminology of kinetic theory), which contradicts the
lower bounds on the measure of the sets above. We thus get the desired
contradiction. The proof is complete.
\end{proof}

\appendix

\section{Isoperimetric lemma implies decrease
of the upper bound}

\begin{proof}[Proof of Lemma~\ref{lem:sup-decrease}]
  We follow the nice exposition of \cite{vasseur}.  Let $C_0$ be the
  universal constant such that solutions $f$ of \eqref{eq:main} in
  $Q_2$ satisfy
\[ \| f_+ \|_{L^\infty(Q_{\frac12})} \le C_0 \|f_+\|_{L^2(Q_1)}.\]
We now define $f_1 =f$ and $f_{k+1} = 2 f_k -1$. Remark that 
\begin{align*}
|\{ f_1 \le 0 \} \cap Q_1 | &\ge \delta_1 \\
\{ f_{k+1} \le 0 \}  &\supset \{ f_k \le 0 \} 
\end{align*}
with $\delta_1 = |Q_1|/2$ (remark it is universal). 
Our goal is to prove that there exists $k_0$ universal such that 
\[ |\{f_{k_0} \ge 0 \} \cap Q_1 | \le \delta_2 \]
with $\delta_2 = (4  C_0^2)^{-1}$ (remark it is universal). 
Indeed, this implies 
\[ \| (f_{k_0})_+ \|_{L^\infty(Q_{\frac12})} \le C_0 \|(f_{k_0})_+\|_{L^2(Q_1)} 
\le C_0 \bigg[ |\{ f_{k_0} \ge 0 \} \cap Q_1 | \bigg]^{\frac12} \le \frac12 \]
which, in turn, yields 
\[ f \le 1 - 2^{-k_0-1} \quad \text{ in } Q_{\frac12}.\]

Assume that for all $k \ge 1$, 
\[ |\{f_k \ge 0 \} \cap Q_1 | \ge \delta_2 .\]
Since $f_{k+1} = 2 f_k -1$, this also implies
\[ |\{f_k \ge \frac12 \} \cap Q_1 | \ge \delta_2 .\]
But we also have 
\[ |\{ f_k \le 0 \} \cap Q_1 | \ge |\{f \le 0 \} \cap Q_1 | \ge \delta_1.\]
Hence Lemma~\ref{lem:isoperim} implies that 
\[ |\{ 0 \le f_k \le \frac12 \} \cap Q_1 | \ge \alpha.\]

Now remark that 
\begin{align*}
|Q_1 | \ge  |\{f_{k+1} \le 0 \} \cap Q_1 | & = |\{f_k \le 0 \} \cap Q_1 | + |\{0 \le f_k \le \frac12 \} \cap Q_1 | \\
& \ge |\{f_k \le 0 \} \cap Q_1 | + \alpha \\
& \ge  k\alpha
\end{align*}
which is impossible for $k$ large enough. 
\end{proof}

\bibliographystyle{acm}
\bibliography{ndgh}

\signci 

\signcm 

\end{document}